\documentclass[preprint]{imsart}

\usepackage{amsthm,amsmath}
\RequirePackage{natbib}
\RequirePackage[colorlinks,citecolor=blue,urlcolor=blue]{hyperref}

\arxiv{}

\startlocaldefs

\usepackage{amssymb,amsfonts,bbm,graphicx,enumitem}
\usepackage[left=3.75cm,top=3.75cm,right=3.75cm,bottom=3.75cm]{geometry}

\newtheorem{theorem}{Theorem}
\newtheorem{lemma}{Lemma}[section]

\newtheorem{corollary}{Corollary}[section]

\theoremstyle{definition}

\theoremstyle{remark}
\newtheorem{remark}{Remark}[section]
\newtheorem{example}{Example}[section]

\makeatletter
\newtheorem*{rep@theorem}{\rep@title}
\newcommand{\newreptheorem}[2]{%
\newenvironment{rep#1}[1]{%
 \def\rep@title{#2 \ref{##1}}%
 \begin{rep@theorem}}%
 {\end{rep@theorem}}}
\makeatother

\newreptheorem{theorem}{Theorem}
\newreptheorem{lemma}{Lemma}
\newreptheorem{corollary}{Corollary}

\numberwithin{equation}{section}

\newcommand{\eps}{\epsilon}

\newcommand{\norm}[1]{\lVert{#1}\rVert}

\def\R{\mathbb{R}}

\newcommand{\bb}{\mathrm{b}}

\renewcommand{\ln}{\log\mathrm{L}} 

\endlocaldefs

\begin{document}

\begin{frontmatter}

  \title{Laplace approximation in 
    high-dimensional Bayesian regression} 
\runtitle{Laplace approximation in high-dimensional regression}

\author{\fnms{Rina Foygel}
  \snm{Barber}\corref{}\ead[label=e1]{rina@uchicago.edu}} 
\address{Department of Statistics\\ The University of Chicago\\
  Chicago, IL 60637, U.S.A.\\ \printead{e1}}
\author{\fnms{Mathias} \snm{Drton}\ead[label=e2]{md5@uw.edu}}
\address{Department of Statistics\\ University of Washington\\
  Seattle, WA 98195, U.S.A.\\ \printead{e2}}
\and 
\author{\fnms{Kean Ming} \snm{Tan}\ead[label=e3]{keanming@uw.edu}}
\address{Department of Biostatistics\\ University of Washington\\
  Seattle, WA 98195, U.S.A.\\ \printead{e3}}


\runauthor{R.F.~Barber, M.~Drton, K.M.~Tan}

\begin{abstract}
  We consider Bayesian variable selection in sparse high-dimensional
  regression, where the number of covariates $p$ may be large relative
  to the sample size $n$, but at most a moderate number $q$ of
  covariates are active.  Specifically, we treat generalized linear
  models.  For a single fixed sparse model with well-behaved prior
  distribution, classical theory proves that the Laplace approximation
  to the marginal likelihood of the model is accurate for sufficiently
  large sample size $n$.  We extend this theory by giving results on
  uniform accuracy of the Laplace approximation across all models
  in a high-dimensional scenario in which $p$ and $q$, and
  thus also the number of considered models, may increase with $n$.
  Moreover, we show how this connection between marginal likelihood
  and Laplace approximation can be used to obtain consistency results
  for Bayesian approaches to variable selection in high-dimensional
  regression.
\end{abstract}





\end{frontmatter}

\section{Introduction}
\label{sec:introduction}

A key issue in Bayesian approaches to model selection is the
evaluation of the marginal likelihood, also referred to as the evidence,
of the different models that are being considered.  While the marginal
likelihood may sometimes be available in closed form when adopting
suitable priors, most problems require approximation techniques.  In
particular, this is the case for variable selection in generalized
linear models such as logistic regression, which are the models
treated in this paper.  Different strategies to approximate the
marginal likelihood are reviewed by \cite{friel:2012}.  Our focus will
be on the accuracy of the Laplace approximation that is derived from
large-sample theory; see also \citet[Section 4.4]{Bishop2006}.

Suppose we have $n$ independent observations of a response variable,
and along with each observation we record a collection of $p$
covariates.  Write $L(\beta)$ for the likelihood function of a
generalized linear model relating the response to the covariates,
where $\beta\in\mathbb{R}^p$ is a vector of coefficients in the linear
predictor \citep{pmcc}.  Let $f(\beta)$ be a prior distribution, and
let $\hat\beta$ be the maximum likelihood estimator (MLE) of the
parameter vector $\beta\in\mathbb{R}^p$.  Then the evidence for the
(saturated) regression model is the integral
\begin{equation*}
  \label{eq:evidence}
  \int_{\mathbb{R}^p} L(\beta) f(\beta) \;d\beta,
\end{equation*}
and the Laplace approximation is the estimate
\begin{equation*}
\label{Eq:posterior}
\text{Laplace} \;:=\;  L (\hat{\beta})
    f(\hat{\beta}) \left(\frac{(2\pi)^{p}}{\det H
        (\hat{\beta})}\right)^{1/2},
\end{equation*}
where $H$ denotes the negative Hessian of the log-likelihood function
$\log L$.  

Classical asymptotic theory for large sample size $n$ but fixed number
of covariates $p$ shows that the Laplace approximation is accurate
with high probability \citep{haughton1988choice}.  With $p$ fixed,
this then clearly also holds for variable selection problems in which
we would consider every one of the finitely many models given by the
$2^p$ subsets of covariates.  This accuracy result justifies the use
of the Laplace approximation as a proxy for an actual model evidence.
The Laplace approximation is also useful for proving
frequentist consistency results about Bayesian methods for variable
selection for a general class of priors.  This is again discussed in
\citet{haughton1988choice}.  The ideas go back to the work of
\citet{schwarz1978estimating} on the Bayesian information criterion
(BIC).

In this paper, we set out to give analogous results on the interplay
between Laplace approximation, model evidence, and frequentist
consistency in variable selection for regression problems that are
high-dimensional, possibly with $p>n$, and sparse in that we consider
only models that involve small subsets of covariates.  We denote 
$q$ as an upper bound on the number of active covariates.  In
variable selection for sparse high-dimensional regression, the number
of considered models is very large, on the order of $p^q$.  Our
interest is then in bounds on the approximation error of Laplace
approximations that, with high probability, hold uniformly across all
sparse models.  Theorem~\ref{thm:LaplaceApproximation}, our main
result, gives such uniform bounds (see Section~\ref{sec:bayesian}).
A numerical experiment supporting the theorem is described in
Section~\ref{subsec:experiment}.

In Section~\ref{sec:consistency}, we show that when adopting suitable
priors on the space of all sparse models, model selection by
maximizing the product of model prior and Laplace approximation is
consistent in an asymptotic scenario in which $p$ and $q$ may grow
with $n$.  As a corollary, we obtain a consistency result for fully
Bayesian variable selection methods.  We note that the class of priors
on models we consider is the same as the one that has been used to
define extensions of BIC that have consistency properties for
high-dimensional variable selection problems \citep[see, for
example,][]{bogdan2004modifying,chen2008extended,zak2011modified,chen2012extended,frommlet2012modified,LuoChen2013,luo2014extended,Rina2014Ising}.
The prior has also been discussed by \cite{scott2010bayes}.


\section{Setup and assumptions}
\label{SetupandAssumptions}

In this section, we provide the setup for the studied problem and the
assumptions needed for our results.

\subsection{Problem setup}
\label{sec:setup}

We treat generalized linear models for $n$ independent observations of
a response, which we denote as $Y_1,\dots,Y_n$.  Each observation
$Y_i$ follows a distribution from a univariate exponential family with
density
\[
p_\theta(y)\,\propto\, \exp\left\{y \cdot
  \theta-\bb(\theta)\right\}\,, \qquad \theta\in\R,
\]
where the density is defined with respect to some measure on $\R$.
Let $\theta_i$ be the (natural) parameter indexing the distribution of
$Y_i$, so $Y_i\sim p_{\theta_i}$.  The vector
$\boldsymbol{\theta}=(\theta_1,\dots,\theta_n)^T$ is then assumed to
lie in the linear space spanned by the columns of a design matrix
$X=(X_{ij})\in\R^{n\times p}$, that is, $\boldsymbol{\theta} = X\beta$
for a parameter vector $\beta\in\R^p$.  Our work is framed in a
setting with a fixed/deterministic design $X$.  In the language of
\cite{pmcc}, our basic setup uses a
canonical link, no dispersion parameter and an exponential family whose
natural parameter space is the entire real line.  This covers, for instance, logistic and Poisson regression.  However, extensions beyond this setting are possible; see for instance the related work of
\cite{LuoChen2013} whose discussion of Bayesian information criteria encompasses
other link functions.

We write $X_i$ for the $i$th row of $X$, that is, the $p$-vector of
covariate values for observation $Y_i$.  The regression model for the
responses then has log-likelihood, score, and negative Hessian
functions
\begin{align*}
  \log L(\beta)&=\sum_{i=1}^n Y_i\cdot X_i^T\beta - 
  \mathrm{b}(X_i^T\beta)\;\in\;\R\;,\\
  s(\beta)&=\sum_{i=1}^n X_i \left(Y_i - 
    \mathrm{b}'(X_i^T\beta)\right)\;\in\;\R^p\;,\\
  H(\beta)&=\sum_{i=1}^n X_iX_i^T\cdot 
  \mathrm{b}''(X_i^T\beta)\;\in\;\R^{p\times p}\;.
\end{align*}
The results in this paper rely on conditions on the Hessian $H$, and
we note that, implicitly, these are actually conditions on the design $X$.  

We are concerned with a sparsity scenario in which the joint
distribution of $Y_1,\dots,Y_n$ is determined by a true parameter
vector $\beta_0\in\R^p$ supported on a (small) set
$J_0\subset[p]:=\{1,\dots,p\}$, 
that is, $\beta_{0j}\not=0$ if and only if $j\in J_0$.  Our interest
is in the recovery of the set $J_0$ when knowing an upper bound $q$ on the
cardinality of $J_0$, so $|J_0|\leq q$.  To this end, we consider the
different submodels given by the linear spaces spanned by subsets
$J\subset[p]$ of the columns of the design matrix $X$, where $|J|\le
q$.  

For notational convenience, we take $J\subset[p]$ to mean either an index
set for the covariates or the resulting regression model.  The regression
coefficients in model $J$ form a vector of length $|J|$.  We index
such vectors $\beta$ by the elements of $J$, that is,
$\beta=(\beta_j:j\in J)$, and we write $\R^J$ for the Euclidean space
containing all these coefficient vectors.  This way the coefficient
and the covariate it belongs to always share a common index.  In other
words, the coefficient for the $j$-th coordinate of covariate vector
$X_i$ is denoted by $\beta_j$ in any model $J$ with $j\in J$.
Furthermore, it is at times convenient to identify a vector
$\beta\in\R^J$ with the vector in $\R^p$ that is obtained from $\beta$
by filling in zeros outside of the set $J$.  As this is clear from the
context, we simply write $\beta$ again when referring to this sparse
vector in $\R^p$.  Finally, $s_J(\beta)$ and $H_J(\beta)$ denote the
subvector and submatrix of $s(\beta)$ and $H(\beta)$, respectively,
obtained by extracting entries indexed by $J$.  These depend only on
the subvectors $X_{iJ}=(X_{ij})_{j\in J}$ of the covariate vectors
$X_i$.

\subsection{Assumptions}
\label{sec:assumptions}

Recall that $n$ is the sample size, $p$ is the number of 
covariates, $q$ is an upper bound on the model size, and $\beta_0$ is the true
parameter vector.  We assume the following conditions to hold for all
considered regression problems: 
\begin{enumerate}[label=(A\arabic*)]
\item \label{ass:A2} The Euclidean norm of the true signal is bounded,
  that is, $\|\beta_0\|_2 \le a_0$ for a fixed constant $a_0\in(0,\infty)$.
\item \label{ass:A3} There is a decreasing function
  $c_{\mathrm{lower}}:[0,\infty)\rightarrow (0,\infty)$ and an
  increasing function $c_{\mathrm{upper}}:[0,\infty)\rightarrow
  (0,\infty)$ such that for all $J\subset[p]$ with $|J|\leq 2q$ and
  all $\beta\in\R^J$, the Hessian of the negative log-likelihood
  function is bounded as
  \begin{equation*}
    \label{A3-1}
    c_{\mathrm{lower}}(\norm{\beta}_2)\mathbf{I}_J\preceq 
    \frac{1}{n}H_J(\beta)\preceq 
    c_{\mathrm{upper}}(\norm{\beta}_2)\mathbf{I}_J.
  \end{equation*}
\item \label{ass:A4} There is a constant
  $c_\text{change}\in(0,\infty)$ such that for all $J\subset[p]$ with $|J|
  \le 2q$ and all $\beta,\beta' \in \mathbb{R}^J$,
\begin{equation*}
\label{A3-2}
\frac{1}{n} \|H_J(\beta) - H_J(\beta')\|_{\text{sp}} \le c_\text{change} \cdot \|\beta-\beta'\|_2,
\end{equation*}
where $\|\cdot\|_{\text{sp}}$ is the spectral norm of a matrix.
\end{enumerate}

Assumption \ref{ass:A3} provides control of the spectrum of the
Hessian of the negative log-likelihood function, and \ref{ass:A4}
yields control of the change of the Hessian.  Together, \ref{ass:A3}
and \ref{ass:A4} imply that for all $\epsilon>0$, there is a $\delta
>0$ such that
\begin{equation}
\label{Eq:A3-3}
(1-\epsilon) H_J(\beta_0) \preceq H_J(\beta_J) \preceq (1+\epsilon)H_J(\beta_0),
\end{equation}
for all $J \supseteq J_0$ with $|J|\le 2q$ and
$\beta_J \in \mathbb{R}^J$ with $\|\beta_J-\beta_0\|_2 \le \delta$;
see Prop.~2.1 in \cite{Rina2014Ising}.  Note also that we consider
sets $J$ with cardinality $2q$ in \ref{ass:A3} and \ref{ass:A4}
because it allows us to make arguments concerning false models, with
$J\not\supseteq J_0$, using properties of the true model given by the
union $J\cup J_0$.

\begin{remark}
  When treating generalized linear models, some
  control of the size of the true coefficient vector $\beta_0$ is
  indeed needed.  For instance, in logistic regression, if the norm of
  $\beta_0$ is too large, then the binary response will take on one of
  its values with overwhelming probability.  Keeping with the setting
  of logistic regression, \cite{Rina2014Ising} show how
  assumptions~\ref{ass:A3} and~\ref{ass:A4} hold with high probability
  in certain settings in which the covariates are generated as
  i.i.d.~sample.  Assumptions~\ref{ass:A3} and~\ref{ass:A4}, or the
  implication from~(\ref{Eq:A3-3}), also appear in earlier work on
  Bayesian information criteria for high-dimensional problems such
  as~\cite{chen2012extended} or \cite{LuoChen2013}.
\end{remark}

Let $\{f_J : J \subset [p], |J|\le q \}$ be a family of probability
density functions $f_J:\mathbb{R}^J\to[0,\infty)$ that we use to
define prior distributions in all $q$-sparse models.  We say that the
family is log-Lipschitz with respect to radius $R>0$ and has bounded
log-density ratios if there exist two constants $ F_1, F_2\in [0,\infty)$
such that the following conditions hold for all $J\subset [p]$ with
$|J|\le q$:
\begin{enumerate}[label=(B\arabic*)]
\item \label{ass:B1} The function $\log f_J$ is $F_1$-Lipschitz on the
  ball $ B_R(0)=\{\beta\in\mathbb{R}^J : \| \beta \|_2 \le R \}$, i.e., for all
  $\beta',\beta \in B_R(0)$, we have
  \[
  |\log f_J (\beta') - \log f_J (\beta)| \le F_1 \|\beta'
  -\beta \|_2.
  \]    
\item \label{ass:B2} For all $\beta \in \mathbb{R}^J$, 
  \[
  \log f_J (\beta) - \log f_J ({0}) \le F_2.
  \]
\end{enumerate}

\begin{example}
  If we take $f_J$ to be the density of a $|J|$-fold product of a
  centered normal distribution with variance $\sigma^2$, then
  \ref{ass:B1} holds with $F_1=R/\sigma^2$ and $F_2=0$.
\end{example}


\section{Laplace approximation}
\label{sec:bayesian}

This section provides our main result.  For a high-dimensional
regression problem, we show that a Laplace approximation to the
marginal likelihood of each sparse model,
\[\mathrm{Evidence}(J):=\int_{\R^J} L (\beta) f_J
                     (\beta) d \beta\;,\]
                      leads to
an approximation error that, with high probability, is bounded uniformly
across all models.  To state our result, we adopt the notation
\[
a=b(1\pm c)
\quad:\!\iff\quad
a\in [b(1-c),b(1+c)].
\]


\begin{theorem}
  \label{thm:LaplaceApproximation}
  Suppose conditions \ref{ass:A2}--\ref{ass:A4} hold.
  Then, there are
  constants $\nu,c_{\mathrm{sample}},a_{\mathrm{MLE}}\in(0,\infty)$
  depending only on
  $(a_0,c_{\mathrm{lower}},c_{\mathrm{upper}},c_{\mathrm{change}})$
  such that if
  \[
  n \ge
  c_{\mathrm{sample}}\cdot  q^3\max\{\log(p),\log^3(n)\} ,
  \]
  then with probability at least $1-p^{-\nu}$ the following two
  statements are true for all sparse models $J\subset[p]$, $|J|\le q$:\smallskip
  
\noindent(i)\quad The MLE $\hat\beta_J$  satisfies $\| \hat{\beta}_J \|_2 \le a_{\mathrm{MLE}}$.\smallskip

\noindent(ii)\quad If additionally the family of prior densities 
  $\{f_J : J \subset [p], |J|\le q \}$ satisfies the
  Lipschitz condition from \ref{ass:B1} for radius $R\geq a_{\mathrm{MLE}}+1$, and has
  log-density ratios bounded as in \ref{ass:B2},
then there is a
    constant $c_{\mathrm{Laplace}}\in(0,\infty)$ depending only on 
    $(a_0,c_{\mathrm{lower}},c_{\mathrm{upper}},c_{\mathrm{change}},F_1,F_2)$
     such that
\[    \mathrm{Evidence}(J) =  L (\hat{\beta}_J)
    f_J(\hat{\beta}_J) \cdot \left(\frac{(2\pi)^{|J|}}{\det H_{J}
        (\hat{\beta}_J)}\right)^{1/2} \cdot \left( 1 \pm
      c_{\mathrm{Laplace}}\sqrt{\frac{|J|^3 \log ^3 
          (n)}{n}} \right).  \]
\end{theorem}


\begin{proof}
  {\em (i) Bounded MLEs.} It follows from
  \citet[Sect.~B.2]{Rina2014Ising}\footnote{In
 the proof of this theorem, we cite several results from \citet[Sect.~B.2, Lem.~B.1]{Rina2014Ising}.
Although that paper treats the specific case of logistic regression, 
by examining the proofs of their results that we
cite here, we can see that they hold more broadly for the general GLM case
as long as we assume that the Hessian conditions hold, i.e.,~Conditions~\ref{ass:A2}--\ref{ass:A4},
and therefore we may use these results for the setting considered here.
} that, with the claimed probability,
  the norms  $\|\hat\beta_J\|_2$ for true models $J$
  (i.e., $J\supseteq J_0$ and $|J|\le 2q$) are bounded by a constant.
  The result makes reference to an event for which all the claims we
  make subsequently are true.  The bound on the norm of an MLE of a
  true model was
  obtained by comparing the maximal likelihood to the likelihood at
  the true parameter $\beta_0$.  As we show now, for false sparse
  models, we may argue similarly but comparing to the likelihood at 0.

  Recall that $a_0$ is the bound on the norm of $\beta_0$
  assumed in~\ref{ass:A2} and that the functions $c_{\mathrm{lower}}$
  and $c_{\mathrm{upper}}$ in~\ref{ass:A3} are decreasing and
  increasing in the norm of $\beta_0$, respectively.  Throughout this
  part, we use the abbreviations
  \[
  c_{\mathrm{lower}} := c_{\mathrm{lower}}(a_0), \qquad
  c_{\mathrm{upper}} := c_{\mathrm{upper}}(a_0).
  \]
  
  First, we lower-bound the likelihood at 0 via a Taylor-expansion
  using the true model $J_0$.  For some $t\in[0,1]$, we have that
  \begin{align*}
    \log L(0) - \log L(\beta_0) 
    &= -\beta_0^Ts_{J_0}(\beta_0) -\frac{1}{2}
      \beta_0^TH_{J_0}(t\beta_0)\beta_0\;\ge\;
      -\beta_0^Ts_{J_0}(\beta_0) -\frac{1}{2} n\, 
      a_0^2c_{\mathrm{upper}} ,
  \end{align*}
  where we have applied~\ref{A3-1}.  Lemma B.1 in
  \cite{Rina2014Ising} yields that  
  \[
  |\beta_0^Ts_{J_0}(\beta_0)| \;\le\;
  \|H_{J_0}(\beta_0)^{-\frac{1}{2}}s_{J_0}(\beta_0)\|
  \|H_{J_0}(\beta_0)^{\frac{1}{2}} \beta_0\|
  \;\le\; \tau_0  a_0\sqrt{n c_{\mathrm{upper}}} ,
  \]
  where $\tau_0^2$ can be bounded by a constant multiple of
  $q\log(p)$.    By our sample size assumption (i.e., the
  existence of the constant $c_{\mathrm{sample}}$), we thus have that
  \begin{equation}
    \label{eq:L0-comparison1}
    \log L(0) - \log L(\beta_0) \;\ge\; -n \cdot c_1
  \end{equation}
  for some constant $c_1\in(0,\infty)$.  

  Second, we may consider the true model $J\cup J_0$ instead of $J$
  and apply (B.17) in \cite{Rina2014Ising} to obtain the bound
  \begin{multline}
    \label{eq:L0-comparison2}
    \log L(\hat\beta_J)-\log L(\beta_0) 
    \le \\
    \|\hat\beta_J-\beta_0\| \cdot\left(
      \sqrt{nc_{\mathrm{upper}}}  \cdot
      \tau_{J\setminus J_0} -\frac{n
        c_{\mathrm{lower}}}{4} \min\left\{
        \|\hat\beta_J-\beta_0\|,
        \frac{c_{\mathrm{lower}}}{2c_{\mathrm{change}}}\right\} \right),
  \end{multline}
  where $\tau_{J\setminus J_0}^2$ can be bounded by a constant multiple
  of $q\log(p)$.  Choosing our sample size constant
  $c_{\mathrm{sample}}$ large enough, we may deduce
  from~(\ref{eq:L0-comparison2}) that there is a constant
  $c_2\in(0,\infty)$ such that
  \[
  \log L(\hat\beta_J) - \log L(\beta_0) \le -n \|\hat\beta_J-\beta_0\|
  c_2
  \]
  whenever
  $\|\hat\beta_J-\beta_0\|>c_{\mathrm{lower}}/(2c_{\mathrm{change}})$.
  Using the fact that $\log L(0)\le \log(\hat\beta_J)$ for any model
  $J$, we may deduce from~(\ref{eq:L0-comparison2}) that there is a
  constant $c_2\in(0,\infty)$ such that
  \[
  \log L(0) - \log L(\beta_0) \le  -n \|\hat\beta_J-\beta_0\| c_2
  \]
  whenever
  $\|\hat\beta_J-\beta_0\|>c_{\mathrm{lower}}/(2c_{\mathrm{change}})$.
  Together with~(\ref{eq:L0-comparison1}), this implies that
  $\|\hat\beta_J-\beta_0\|$ is bounded by a constant $c_3$.  Having
  assumed \ref{ass:A2}, we may conclude that the norm of $\hat\beta_J$
  is bounded by $a_0+c_3$.

  \medskip
  
  {\em (ii) Laplace approximation.}  Fix $J\subset [p]$ with
  $|J|\le q$.  In
  order to analyze the evidence of model $J$, we split the integration
  domain $\mathbb{R}^J$ into two regions, namely, a neighborhood
  $\mathcal{N}$ of the MLE $\hat{\beta}_J$ and the complement
  $\mathbb{R}^J \backslash \mathcal{N}$.  More precisely, we choose
  the neighborhood of the MLE as
  \[
  \mathcal{N} := \left\{\, \beta\in\mathbb{R}^J \::\: \|H_{J}
    (\hat{\beta}_J)^{1/2} (\beta-\hat{\beta}_J)\|_2 \le \sqrt{5 |J|
      \log (n)} \,\right\}.
  \]
  Then the marginal likelihood, $\text{Evidence}(J)$, is the sum of the
  two integrals
  \begin{align*}
    \mathcal{I}_1 &=\int_{\mathcal{N}} L (\beta) f_J
                     (\beta) d \beta,\\
    \mathcal{I}_2 &=\int_{ \mathbb{R}^J \backslash
                     \mathcal{N}} L (\beta) f_J (\beta) d
                     \beta.
  \end{align*}
  We will estimate $\mathcal{I}_1$ via a quadratic approximation to
  the log-likelihood function.  Outside of the region $\mathcal{N}$,
  the quadratic approximation may no longer be accurate but due to
  concavity of the log-likelihood function, the integrand can be
  bounded by $e^{-c \|\beta_J - \hat{\beta}_J\|_2}$ for an
  appropriately chosen constant $c$, which allows us to show that
  $\mathcal{I}_2$ is negligible when $n$ is sufficiently large.

  We now approximate $\mathcal{I}_1$ and $\mathcal{I}_2$ separately.
  Throughout this part we assume that we have a bound
  $a_{\mathrm{MLE}}$ on the norms of the MLEs $\hat\beta_J$ in sparse
  models $J$ with $|J|\le q$.  For notational convenience, we now let
  \[
  c_{\mathrm{lower}} := c_{\mathrm{lower}}(a_{\mathrm{MLE}}), \qquad
  c_{\mathrm{upper}} := c_{\mathrm{upper}}(a_{\mathrm{MLE}}).
  \]

  {\em (ii-a) Approximation of integral $\mathcal{I}_1$.}  By a
  Taylor-expansion, for any $\beta\in\mathbb{R}^J$ there is a
  $t\in[0,1]$ such that
  \begin{align*}
    \log L(\beta)&= \log L(\hat\beta_J) 
     - \frac{1}{2} (\beta-\hat{\beta}_J)^T H_{J}\left(\hat{\beta}_J + t
      (\beta-\hat{\beta}_J)\right) (\beta-\hat{\beta}_J).
  \end{align*}
  By~\ref{ass:A4} and using that $|t|\le 1$,
  \begin{equation*}
    \left\|H_{J}\left(\hat{\beta}_J + t
        (\beta-\hat{\beta}_J)\right)-H_J(\hat{\beta}_J)\right\|_{\mathrm{sp}}
    \;\le\; 
    n \cdot c_{\mathrm{change}} \, \|\beta-\hat\beta_J\|_2.
  \end{equation*}
  Hence, 
  \begin{align}
        \label{Eq:mvtapproximation0}
    \log L(\beta) &= \log L(\hat\beta_J) - \frac{1}{2}
                    (\beta-\hat{\beta}_J)^T H_{J}(\hat{\beta}_J) 
  (\beta-\hat{\beta}_J) \pm \frac{1}{2}
  \|\beta-\hat{\beta}_J\|_2^3 \cdot n \,c_{\mathrm{change}}.
  \end{align}
  Next, observe that \ref{ass:A3} implies that
  \[
  H_{J}(\hat{\beta}_J)^{-1/2} \;\preceq\; \sqrt{\frac{1}{n
      c_{\mathrm{lower}}}} \cdot \mathbf{I}_J.
  \]
  We deduce that for any vector $\beta\in\mathcal{N}$, 
  \begin{equation}
    \label{Eq:beta bound}
    \|\beta -\hat{\beta}_J \|_2 \;\le\;   \sqrt{5|J|\log (n)}\, \|H_{J}
    (\hat{\beta}_J)^{-1/2}\|_{\mathrm{sp}} \;\le\; \sqrt{\frac{5|J| \log
        (n)}{nc_{\mathrm{lower}} }}. 
  \end{equation}
  This gives
  \begin{align}
        \label{Eq:mvtapproximation}
    \log L(\beta) 
    &= \log L(\hat\beta_J) - \frac{1}{2}
      (\beta-\hat{\beta}_J)^T H_{J}(\hat{\beta}_J) 
      (\beta-\hat{\beta}_J) 
      \pm\sqrt{\frac{|J|^3 \log^3 (n)}{n}}\cdot \sqrt{\frac{125
      c^2_{\mathrm{change}}}{4 c^3_{\mathrm{lower}}}}. 
  \end{align}

  Choosing the constant $c_{\mathrm{sample}}$ large enough, we can
  ensure that the upper bound in~(\ref{Eq:beta bound}) is no larger
  than 1.  In other words, $\|\beta-\hat{\beta}_J\|_2 \le 1$ for all
  points $\beta \in \mathcal{N}$.  By our assumption that
  $\|\hat{\beta}_J \|_2 \le a_{\mathrm{MLE}}$, the set $\mathcal{N}$
  is thus contained in the ball
  \[
  \mathcal{B}=\left\{\beta\in\mathbb{R}^J: \|\beta\|_2 \le
  a_{\mathrm{MLE}}+1\right\}.
  \]
  Since, by~\ref{ass:B1}, the logarithm of the prior density is
  $F_1$-Lipschitz on $\mathcal{B}$, it follows form~(\ref{Eq:beta bound}) that
  \begin{equation}
    \label{Eq:prior_bound}
    \log f_J (\beta) \;=\; \log f_J (\hat{\beta}_J)  \pm F_1
    \|\beta-\hat{\beta}_J\|_2 \;=\; \log f_J (\hat{\beta}_J)  \pm F_1
    \sqrt{\frac{5|J| \log (n)}{nc_{\mathrm{lower}} }}. 
  \end{equation}

  Plugging~(\ref{Eq:mvtapproximation}) and~(\ref{Eq:prior_bound}) into
  $\mathcal{I}_1$, and writing $a=b\cdot \exp\{\pm c\}$ to denote $a\in[b\cdot e^{-c},b\cdot e^c]$, we find that
  \begin{equation}
    \label{eq:laplace1}
    \begin{split}
      \mathcal{I}_1 &= L(\hat\beta_J)f_J (\hat{\beta}_J) \exp \left\{
        \pm\sqrt{\frac{|J|^3 \log^3 (n)}{n}}\cdot \left(
          \sqrt{\frac{5F_1^2}{nc_{\mathrm{lower}}}}
          +\sqrt{\frac{125 c^2_{\mathrm{change}}}{4
              c^3_{\mathrm{lower}}}} \right)    \right\} \\
      &\qquad \qquad \times \int_{\mathcal{N}} \exp \left\{   -
        \frac{1}{2} (\beta-\hat{\beta}_J)^T H_{J}(\hat{\beta}_J)
        (\beta-\hat{\beta}_J)  \right\} d\beta. 
    \end{split}
  \end{equation}
  In the last integral, change variables to
  $\xi= H_J (\hat{\beta}_J)^{1/2} (\beta-\hat{\beta}_J)$ to see that
  \begin{align}
    \nonumber
    \int_{\mathcal{N}} \exp& \left\{   - \frac{1}{2}
                             (\beta-\hat{\beta}_J)^T H_{J}(\hat{\beta}_J)
                             (\beta-\hat{\beta}_J)  \right\} d\beta
    \\
    \nonumber
                           &=  \left(\det   H_J(\hat{\beta}_J)
                             \right)^{-1/2} \cdot  \int_{ \|\xi\|_2
                             \le \sqrt{5|J| \log (n)}} \exp \left\{   -
                             \frac{1}{2} \|\xi\|^2_2  \right\} d\xi\\
    \nonumber
                           &= \left(\frac{(2\pi)^{|J|}}{\det
                             H_J(\hat{\beta}_J)}    \right)^{1/2}
                             \cdot \Pr\left\{\chi^2_{|J|}\le
                             5|J|\log(n)\right\}\\
    \label{eq:gaussian-integral}
                           &= \left(\frac{(2\pi)^{|J|}}{\det
                             H_J(\hat{\beta}_J)}    \right)^{1/2}
                             \cdot \exp\{\pm1/\sqrt{n}\},
  \end{align}
  where we use a tail bound for the $\chi^2$-distribution stated in
  Lemma~\ref{lemma:chi-square tail bound}.  We now
  substitute~(\ref{eq:gaussian-integral}) into~(\ref{eq:laplace1}), and
  simplify the result using that $e^{-x} \ge 1-2 x$ and $e^x \le 1+2x$
  for all $0 \le x \le 1$.  We find that
  \begin{equation}
    \label{eq:I1-bound}
  \mathcal{I}_1 =L(\hat{\beta}_J) f_J(\hat{\beta}_J)
\left(\frac{(2\pi)^{|J|}}{\det H_{J}
        (\hat{\beta}_J)}\right)^{1/2}  \left( 1 \pm
    2\left(1+\sqrt{\frac{125 c_{\mathrm{change}}^2 }{4
          c_{\mathrm{lower}}^3}} +
      \sqrt{\frac{5F_1^2}{c_{\mathrm{lower}}}} \right) \sqrt{\frac{|J|^3
        \log ^3 (n)}{n}} \right) 
  \end{equation}
  when the constant $c_{\mathrm{sample}}$ is chosen large enough.
  \smallskip

  {\em (ii-b) Approximation of integral $\mathcal{I}_2$.}  Let $\beta$
  be a point on the boundary of $\mathcal{N}$.  It then holds that
  \begin{align*}
    (\beta-\hat{\beta}_J)^T H_J (\hat{\beta}_J) (\beta-\hat{\beta}_J) 
    &=\sqrt{5|J| \log (n)}\cdot \|H_J (\hat{\beta}_J)^{1/2}
      (\beta-\hat{\beta}_J)\|_2  .
  \end{align*}
  We may deduce from 
  (\ref{Eq:mvtapproximation}) that
  \begin{align*}
    \nonumber
      \log L (\beta)  &\le \log L (\hat{\beta}_J) -\frac{\sqrt{5|J| \log (n)}}{2} \|H_J (\hat{\beta}_J)^{1/2}
      (\beta-\hat{\beta}_J)\|_2+
      \sqrt{\frac{|J|^3 \log^3 (n)}{n}}\cdot \sqrt{\frac{125
          c^2_{\mathrm{change}}}{4 c^3_{\mathrm{lower}}}}  \\ 
      \label{asdf}
      &\le \log L (\hat{\beta}_J)  - \|H_J (\hat{\beta}_J)^{1/2} (\beta-\hat{\beta}_J)\|_2 \cdot \sqrt{|J|\log (n)},
  \end{align*}
  for $ |J|^3 \log^3 (n)/n$ sufficiently small, which can be ensured
  by choosing $c_{\mathrm{sample}}$ large enough.  The concavity of
  the log-likelihood function now implies that for all
  $\beta\not\in\mathcal{N}$ we have
  \begin{equation}
    \label{eq:ii-b-like}  
    \log L (\beta) \le \log L(\hat{\beta}_J) - \|H_J
    (\hat{\beta}_J)^{1/2} (\beta-\hat{\beta}_J)\|_2 \cdot \sqrt{|J|\log
      (n)}. 
  \end{equation}
  Moreover, using first assumption~\ref{ass:B2} and then
  assumption~\ref{ass:B1}, we have that
  \begin{equation*}
    \log f_J(\beta) \;\le\; \log f_J ({0}) + F_2 \;\le\; \log
    f_J(\hat{\beta}_J) + F_1 \|\hat{\beta}_J\|_2 + F_2.
  \end{equation*}
  Since $\|\hat{\beta}_J\|_2\le a_{\mathrm{MLE}}$, it thus holds that  
  \begin{equation}
    \label{eq:ii-b-prior}  
      \log f_J(\beta) \;\le\; \log
      f_J(\hat{\beta}_J) + F_1 a_{\mathrm{MLE}}  + F_2.
  \end{equation}

  Combining the bounds from~(\ref{eq:ii-b-like})
  and~(\ref{eq:ii-b-prior}), the integral can be bounded as
  \begin{equation}
    \label{eq:ii-b-bound1}
    \mathcal{I}_2 
    \;\le\; L(\hat{\beta}_J)f_J (\hat{\beta}_J) e^{ F_1
      a_{\mathrm{MLE}} + F_2} \cdot \int_{\mathbb{R}^J\setminus
      \mathcal{N}} \exp \left\{ -\|H_J(\hat{\beta}_J)^{1/2}
      (\beta-\hat{\beta}_J)  \|_2  \cdot \sqrt{|J|\log (n)} \right\}
    d\beta. 
  \end{equation}
  Changing variables to
  $\xi = H_J (\hat{\beta}_J)^{1/2} (\beta-\hat{\beta}_J)$ and
  applying Lemma~\ref{lemma:integral bound}, we may bound the integral
  in~(\ref{eq:ii-b-bound1}) as
  \begin{align*}
    &\int_{\mathbb{R}^J\setminus
      \mathcal{N}} \exp \left\{ -\|H_J(\hat{\beta}_J)^{1/2}
      (\beta-\hat{\beta}_J)  \|_2  \cdot \sqrt{|J|\log (n)} \right\}
      d\beta \\
    &\le \left(\det H_J(\hat{\beta}_J)\right)^{-1/2}  \cdot
      \int_{\|\xi\|_2 > \sqrt{5|J|\log(n)}} \exp \left\{
      -\sqrt{|J|\log (n)} \cdot \|\xi\|_2 \right\} d\xi \\ 
    &\le \left(\det H_J(\hat{\beta}_J)\right)^{-1/2}  \cdot  
      \frac{4(\pi)^{|J|/2}}{\Gamma\left( \frac{1}{2} |J|\right)}
      \frac{\sqrt{5|J|\log(n)}^{|J|-1}}{\sqrt{|J|\log (n)}} e^{-\sqrt{5} \,
      |J|\log(n)}\\
    &= \left(\frac{(2\pi)^{|J|}}{\det
      H_J(\hat{\beta}_J)}    \right)^{1/2}
      \cdot
      \frac{2\sqrt{5}}{\Gamma\left( \frac{1}{2} |J|\right)}
      \left( \frac{5}{2}|J|\log(n)\right)^{|J|/2-1}\cdot \frac{1}{n^{\sqrt{5} \,
      |J|}}.
  \end{align*}
  Stirling's lower bound on the Gamma function gives
  \begin{align*}
      \frac{(|J|/2)^{|J|/2-1}}{\Gamma\left( \frac{1}{2} |J|\right)}
    \;=\;
      \frac{(|J|/2)^{|J|/2}}{\Gamma\left( \frac{1}{2} |J|+1\right)}
    \;\le\;
    \frac{1}{\sqrt{|J|\pi}}e^{|J|/2}.
  \end{align*}
  Using this inequality, and returning to (\ref{eq:ii-b-bound1}), we
  see that
  \begin{multline}
    \label{eq:I2-integral-bound}
    \mathcal{I}_2 \;\le\; L(\hat{\beta}_J)f_J (\hat{\beta}_J)
    \left(\frac{(2\pi)^{|J|}}{\det
        H_J(\hat{\beta}_J)}    \right)^{1/2}\\
    \times e^{ F_1 a_{\mathrm{MLE}} + F_2} \cdot
    \frac{2e\sqrt{5}}{\sqrt{|J|\pi} } \cdot \left(\frac{5e\log(n)}{n}
    \right)^{|J|/2-1} \frac{1}{n^{(\sqrt{5}-1/2)\cdot |J|+1}}.
  \end{multline}
  Based on this fact, we certainly have the very loose bound that 
  \begin{equation}
    \label{eq:I2-bound}
  \mathcal{I}_2 \;\le\; L(\hat{\beta}_J) f_J (\hat{\beta}_J)
  \left(\frac{(2\pi)^{|J|}}{\det H_J(\hat{\beta}_J)} \right)^{1/2}
   e^{F_1a_{\mathrm{MLE}}+F_2}\cdot \frac{1}{\sqrt{n}},
  \end{equation}
  for all sufficiently large $n$.
  \smallskip

  {\em (ii-c) Combining the bounds.}  From~(\ref{eq:I1-bound})
  and~(\ref{eq:I2-bound}), we obtain that
  \begin{multline}
    \label{eq:laplace-final}
    \text{Evidence}(J) = \mathcal{I}_1+\mathcal{I}_2 
    =L(\hat{\beta}_J) f_J(\hat{\beta}_J) 
      \left(\frac{(2\pi)^{|J|}}{\det H_J(\hat{\beta}_J)}
      \right)^{1/2}\times 
    \\
                       \qquad  \left( 1 \pm
                          \left(e^{F_1a_{\mathrm{MLE}}+F_2}+2+\sqrt{\frac{125
                          c_{\mathrm{change}}^2 }{ c_{\mathrm{lower}}^3}} +
                          \sqrt{\frac{20F_1^2}{c_{\mathrm{lower}}}}
                          \right) \sqrt{\frac{|J|^3 
                          \log ^3 (n)}{n}} \right)
  \end{multline}
  for sufficiently large $n$, as desired.
\end{proof}

\begin{remark}
  The proof of Theorem~\ref{thm:LaplaceApproximation} could be
  modified to handle other situations of interest.  For instance,
  instead of a fixed Lipschitz constant $F_1$ for all log prior
  densities, one could consider the case where $\log f_J$ is Lipschitz
  with respect to a constant $F_1(J)$ that grows with the cardinality
  of $|J|$, e.g., at a rate of $\sqrt{|J|}$ in which case the rate of
  square root of $|J|^3\log^3(n)/n$ could be modified to square root
  of $|J|^4\log^3(n)/n$.  The term $e^{F_1(J) a_{\mathrm{MLE}}}$ that
  would appear in~(\ref{eq:laplace-final}) could be compensated
  using~(\ref{eq:I2-integral-bound}) in less crude of a way than when
  moving to~(\ref{eq:I2-bound}).
\end{remark}

\section{Numerical experiment for sparse Bayesian logistic regression}
\label{subsec:experiment}

In this section, we perform a simulation study to assess the approximation error in Laplace approximations to the marginal likelihood of logistic regression models.   To this end, we generate independent covariate vectors $X_1,\dots,X_n$ with i.i.d.~$N(0,1)$ entries.  For each choice of a (small) value of $q$, we take the true parameter vector $\beta_0\in\mathbb{R}^p$ to have the first $q$ entries equal to two and the rest of the entries equal zero.  So, $J_0=\{1,\dots,q\}$.  We then generate independent binary responses $Y_1,\dots,Y_n$, with values in $\{0,1\}$ and distributed as $(Y_i |X_i) \sim \mathrm{Bernoulli}(p_i (X_i) ) $, where
\[
p_i (x) = \left( \frac{\exp \left( x^T\beta_0 \right) }{1+\exp \left( x^T\beta_0 \right)}\right) \iff \log \left( \frac{p_i (x)}{1-p_i (x)} \right) = x \cdot  \beta_0,
\]
based on the usual (and canonical) logit link function.  

We record that the logistic regression model with covariates indexed by $J\subset [p]$ has 
the likelihood function
\begin{equation}
\label{Eq:likelihood}
L(\beta) =  \exp \left\{ \sum_{i=1}^n Y_i \cdot X_{iJ}^T\beta -  \log \left(1+\exp (X_{iJ}^T \beta) \right)\right\}, \quad \beta\in\mathbb{R}^J,
\end{equation}
where, as previously defined, $X_{iJ} = (X_{ij})_{j\in J}$ denotes the subset of covariates for model $J$.   The negative Hessian of the log-likelihood function is
\[
H_J (\beta) = \sum_{i=1}^n X_{iJ} X_{iJ}^T \cdot \frac{\exp (X_{iJ}^T \beta) }{\left(1+\exp (X_{iJ}^T \beta)\right)^2}.
\]
For Bayesian inference in the logistic regression model given by $J$, we consider as a prior distribution a standard normal distribution on $\mathbb{R}^J$, that is, the distribution of a random vector with $|J|$ independent $N(0,1)$ coordinates.  As in previous section, we denote the resulting prior density by $f_J$. We then wish to approximate the evidence or marginal likelihood
\[
 \text{Evidence}(J) \;:=\; \int_{\mathbb{R}^J} L(\beta) f_{J}(\beta) \;d\beta.
\]

As a first approximation, we use a Monte Carlo approach in which we simply draw independent samples $\beta^1,\dots,\beta^B$ from the prior $f_J$ and estimate the evidence as
\[
\mathrm{MonteCarlo}(J) =\frac{1}{B} \sum_{b=1}^B  L (\beta^b),
\]
where we use $B=50,000$ in all of our simulations.  As a second method, we compute the Laplace approximation
\[
\mathrm{Laplace}(J) := L(\hat{\beta}_J) f_J(\hat{\beta}_J) \left(\frac{(2\pi)^{|J|}}{\det H_J(\hat{\beta}_J)}
      \right)^{1/2},
\]
where $\hat{\beta}_J$ is the maximum likelihood estimator in model $J$.
For each choice of the number of covariates $p$, the model size $q$, and the sample size $n$, we calculate the Laplace approximation error as
\[
\max_{J\subset[p],\, |J|\le q} \;  \left|\, \log  \mathrm{MonteCarlo}(J) - \log  \mathrm{Laplace}(J)\,\right|.
\]

We consider $n\in \{50,60,70,80,90,100\}$ in our experiment.  Since we wish to compute the Laplace approximation error of every $q$-sparse model, and the number of possible models is on the order of $p^q$, we consider $p=n/2$ and $q\in \{1,2,3\}$.  The Laplace approximation error, averaged across 100 independent simulations, is shown in Figure~\ref{Figure:simresult}.  We remark that the error in the Monte Carlo approximation to the marginal likelihood is negligible compared to the quantity plotted in Figure~\ref{Figure:simresult}.  With two independent runs of our Monte Carlo integration routine, we found the Monte Carlo error to be on the order of 0.05.               

For each $q=1,2,3$, Figure~\ref{Figure:simresult} shows a decrease in Laplace approximation error as $n$ increases.  We emphasize that $p$ and thus also the number of considered $q$-sparse models increase with $n$.    As we increase the number of active covariates $q$, the Laplace approximation error increases.  These facts are in agreement with Theorem~\ref{thm:LaplaceApproximation}.   This said, the scope of this simulation experiment is clearly limited by the fact that only small values of $q$ and moderate values of $p$ and $n$ are computationally feasible.

\begin{figure}[t]
\includegraphics[scale=0.5]{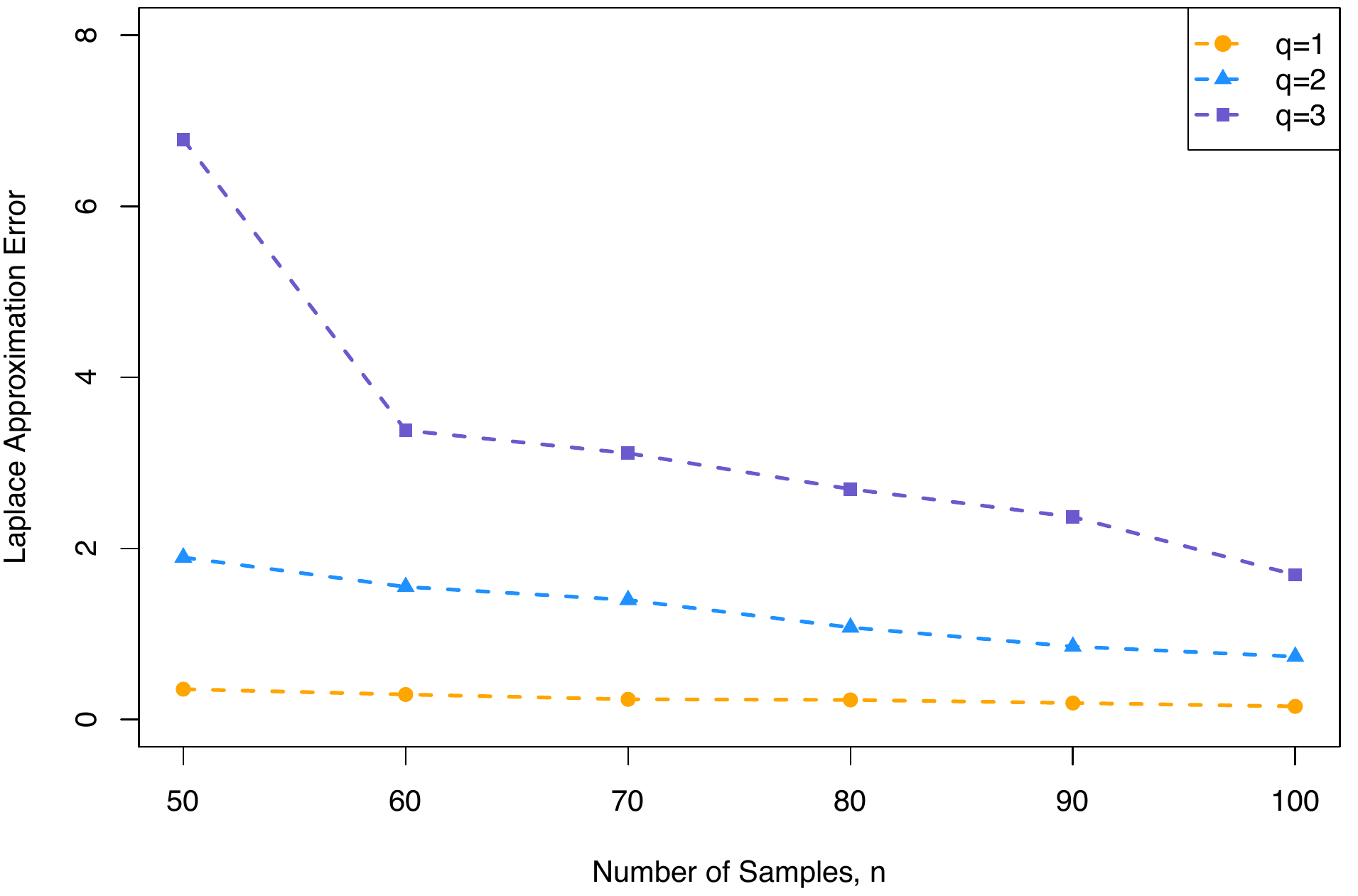}
\caption{\label{Figure:simresult} Maximum Laplace approximation error, averaged over 100 data sets, as a function of the sample size $n$.  The number of covariates is $n/2$, and the number of active covariates is bounded by $q\in \{1,2,3\}$.  }
\end{figure}


\section{Consistency of Bayesian variable selection}
\label{sec:consistency}

In this section, we apply the result on uniform accuracy of the Laplace approximation (Theorem~\ref{thm:LaplaceApproximation}) to prove a high-dimensional consistency result for Bayesian variable selection.  Here, consistency refers to the property that the probability of choosing the most parsimonious true model tends to one in a large-sample limit.   As discussed in the Introduction, we consider priors of the form 
\begin{equation}
\label{Eq:prior}
P_\gamma(J) \propto {p\choose |J|}^{-\gamma} \cdot \mathbbm{1}\{|J| \le q \}, \qquad J\subset[p], 
\end{equation}
where $\gamma \ge 0$ is a parameter that allows one to interpolate between the case of a uniform distribution on models ($\gamma=0$) and a prior for which the model cardinality $|J|$ is uniformly distributed ($\gamma=1$).  

Bayesian variable selection is based on maximizing the (unnormalized) posterior probability
\begin{equation}
 \text{Bayes}_{\gamma}(J) := {p\choose |J|}^{-\gamma} \text{Evidence}{(J)}
\end{equation}	
over $J\subset[p]$, $|J|\le q$.  Approximate Bayesian variable section can be based on maximizing instead the quantity
\begin{equation}
 \text{Laplace}_{\gamma}(J) := {p\choose |J|}^{-\gamma} \text{Laplace}{(J)}.
\end{equation}
We will identify asymptotic scenarios under which maximization of $\text{Laplace}_{\gamma}$ yields consistent variable selection.  Using Theorem~\ref{thm:LaplaceApproximation},  we obtain as a corollary that fully Bayesian variable selection, i.e., maximization of $\text{Bayes}_{\gamma}$, is consistent as well.

To study consistency, we consider a sequence of variable selection problems indexed by the sample size $n$, where the $n$-th problem has $p_n$ covariates, true parameter $\beta_0 (n)$ with support $J_0(n)$, and signal strength
\[
\beta_{\min} (n) = \underset{j\in J_0(n)}{\min} |(\beta_0(n))_j|.
\]
In addition, let $q_n$ be the upper bound on the size of the considered models.  The following consistency result is similar to the related results for extensions of the Bayesian information criterion \citep[see, for instance,][]{chen2012extended,Rina2014Ising}. 

\begin{theorem}
\label{thm:consistency}
Suppose that $p_n=n^\kappa$ for $\kappa >0$, that $q_n = n^\psi$ for
$0\le \psi < 1/3$, that $\beta_{\min} (n) = n^{-\phi/2}$ for $0 \le
\phi < 1-\psi$, and that $\kappa > \psi$.  Assume that \ref{ass:A2}
holds for a fixed constant $a_0$ and that there a fixed functions $c_{\mathrm{lower}}$ and $c_{\mathrm{upper}}$ with respect to which the covariates satisfy the Hessian
conditions~\ref{ass:A3} and~\ref{ass:A4} for all $J \supseteq J_0(n)$
with $|J| \le 2 q_n$.  Moreover, assume that for the considered family of prior densities $\{ f_J (\cdot) : J \subset [p_n], |J| \le q_n\}$ there are constants $ F_3, F_4 \in(0,\infty)$ such
that, uniformly for all $|J| \le q_n$, we have
\[
\sup_{\beta} f_J(\beta) \le F_3 < \infty,  \qquad
\inf_{\|\beta\|_2 \le a_{\mathrm{MLE}}} f_J (\beta) \ge F_4 > 0,
\]
where $a_{\mathrm{MLE}}$ is the constant from Theorem~\ref{thm:LaplaceApproximation}(i).
 Then, for any $\gamma > 1- \frac{1-2\psi}{2(\kappa-\psi)}$, model selection with $\mathrm{Laplace}_{\gamma}$ is consistent in the sense that the event 
\[
J_0(n) = \arg \max \{ \mathrm{Laplace}_{\gamma}(J) : J\subset [p_n], |J| \le q_n \}
\]
has probability tending to one as $n\rightarrow \infty$.
\end{theorem}

Together with Theorem~\ref{thm:LaplaceApproximation}, the proof of
Theorem~\ref{thm:consistency}, which we give below, also shows
consistency of the fully Bayesian procedure.

\begin{corollary}
  \label{cor:bayes-consistency}
  Under the assumptions of Theorem~\ref{thm:consistency},
  fully Bayesian model selection is consistent, that is, the event
  \[
J_0(n) = \arg \max \{
  \mathrm{Bayes}_{\gamma}(J) : J\subset [p_n], |J| \le q_n \}
  \]
  has probability tending to one as $n\rightarrow \infty$.
\end{corollary}

\begin{proof}[Proof of Theorem~\ref{thm:consistency}]
Our scaling assumptions for $p_n, q_n$ and $\beta_{\min} (n)$ are such
that the conditions imposed in Theorem 2.2 of \cite{Rina2014Ising} are
met for $n$ large enough.  This theorem and
Theorem~\ref{thm:LaplaceApproximation}(i) in this paper then yield that there are constants $\nu,\epsilon,C_\mathrm{false},a_{\mathrm{MLE}}>0$ such that with probability at least $1-p_n^\nu$ the following three statements hold simultaneously:
\begin{enumerate}[label=(\alph*)]
  \item \label{thm2:a} For all $|J|\leq q_n$ with
    $J\supseteq J_0(n)$,
    \begin{equation}
      \label{eq:consist-fact-a}
    \ln(\hat{\beta}_J)-\ln(\hat{\beta}_{J_0(n)})\;\leq\;
    (1+\eps)(|J\backslash J_0(n)|+\nu)\log(p_n)\;.
    \end{equation}
  \item \label{thm2:b} For all $|J|\leq q_n$ with
    $J\not\supset J_0(n)$,
    \begin{equation}
      \label{eq:consist-fact-b}
    \ln(\hat{\beta}_{J_0}(n))-\ln(\hat{\beta}_J)\;\geq\;
    C_\mathrm{false} \,n\, \beta_{\min} (n)^2\;.
    \end{equation}
 \item For all $|J|\le q_n$ and some constant $a_{\mathrm{MLE}}>0$,
   \begin{equation}
     \label{eq:consist-fact-c}
     \|\hat{\beta}_J\| \le a_{\mathrm{MLE}}.
   \end{equation}
  \end{enumerate}
In the remainder of this proof we show that these three facts, in combination with further technical results from \cite{Rina2014Ising}, imply that 
\begin{equation}
  \label{eq:consist-claim}
  J_0(n) = \arg \max \{ \mathrm{Laplace}_{\gamma}(J) : J\subset [p_n], |J| \le q_n \}.
\end{equation}
For simpler notation, we no longer indicate explicitly that $p_n$,
$q_n$, $\beta_0$ and derived quantities vary with $n$.  We will then
show that 
\begin{multline}
\label{Eq:posteriorthm2}
\log \frac{\text{Laplace}_{\gamma}(J_0)}{\text{Laplace}_{\gamma}(J)}  = \left( \log P(J_0)-\log P(J)\right) +\left( \log L (\hat{\beta}_{J_0}) - \log L (\hat{\beta}_{J}) \right)- |J \backslash J_0| \log \sqrt{2\pi}\\
+  \left(  \log f_{J_0} (\hat{\beta}_{J_0})- \log f_J
  (\hat{\beta}_J)\right) + \frac{1}{2} \left( \log \det H_J (\hat{\beta}_J) -\log \det H_{J_0} (\hat{\beta}_{J_0}) \right)
\end{multline}
is positive for any model given by a set $J \ne J_0$ of cardinality
$|J| \le q$.  We let
\[
c_{\mathrm{lower}} := c_{\mathrm{lower}}(a_{\mathrm{MLE}}), \qquad
c_{\mathrm{upper}} := c_{\mathrm{upper}}(a_{\mathrm{MLE}}).
\]

\smallskip
{\em False models.}  If $J \not \supseteq J_0$, that is, if the model is false, we observe that 
\[
\log P(J_0)-\log P(J) = -\gamma \log {p \choose |J_0|}+ \gamma \log {p \choose |J|}\ge -\gamma \log {p\choose |J_0|} \ge  -\gamma q\log p.
\]
Moreover, by \ref{ass:A3}
and~(\ref{eq:consist-fact-c}),
\begin{align*}
\log \det H_J (\hat{\beta}_J) -\log \det H_{J_0} (\hat{\beta}_{J_0})
&\ge |J| \cdot\log(nc_{\mathrm{lower}}) -
|J_0| \cdot\log(nc_{\mathrm{upper}}) \\
&\ge - q \log \left(n
  \frac{c_{\mathrm{upper}}}{\min\{c_{\mathrm{lower}},1\}} \right).
\end{align*}
Combining the lower bounds with~(\ref{eq:consist-fact-b}), we obtain that
\begin{align*}
  \log \frac{\text{Laplace}_{\gamma}(J_0)}{\text{Laplace}_{\gamma}(J)}
  &\ge  C_{\mathrm{false}} n \beta_{\min}^2 -|J\backslash J_0|
    \log (\sqrt{2\pi}) - q \log \left(p^{\gamma} n\frac{c_{\mathrm{upper}}}{\min\{c_{\mathrm{lower}},1\}}\right) + \log \left(\frac{F_4}{F_3}\right)\\
  &\ge  C_{\mathrm{false}} n \beta_{\min}^2 -q \log
    \left(\frac{c_{\mathrm{upper}}}{\min\{c_{\mathrm{lower}},1\}} \cdot
    \sqrt{2\pi} n p^{\gamma}\right)+ \log 
    \left(\frac{F_4}{F_3}\right). 
\end{align*}
By our scaling assumptions, the lower bound is positive for
sufficiently large $n$.  

\smallskip {\em True models.}  It remains to resolve the case of
$J \supsetneq J_0$, that is, when model $J$ is true.  We record that
from the proof of Theorem~2.2 in \citet{Rina2014Ising}, it holds on
the considered event of probability at least $1-p^{-\nu}$ that for any
$J\supseteq J_0$,
\begin{equation}
  \label{eq:consist-true-MLE}
\|\hat{\beta}_J -\beta_0\|_2 \le  \frac{4 \sqrt{c_{\mathrm{upper}} }}{\sqrt{n} c_{\mathrm{lower}} } \tau_{|J\backslash J_0|},
\end{equation}
where 
\[
\tau_{r}^2 = \frac{2}{(1-\epsilon')^3} \cdot \left[ (J_0+r) \log
  \left( \frac{3}{\epsilon'} \right) + \log(4p^\nu) + r\log (2p)
\right]. 
\] 
Under our scaling assumptions on $p$ and $q$, it follows that
$\|\hat{\beta}_J -\beta_0\|_2$ tends to zero as $n\to\infty$.

We begin again by considering the prior on models, for which we have that
\begin{equation*}
\begin{split}
\log P(J_0) - \log P(J) = \gamma \log \frac{|J_0|! (p-|J_0|)!}{|J|! (p-|J|)!}
&\;\ge\;   -\gamma |J\backslash J_0| \log q + \gamma |J\backslash J_0| \log (p-q)\\
&\;\ge\;   -\gamma |J\backslash J_0| \log q + \gamma |J\backslash J_0| (1-\tilde{\epsilon})  \log p
\end{split}
\end{equation*}
for all $n$ sufficiently large.  Indeed, we assume that
$p = n^\kappa$ and $q=n^\psi$ with $\kappa > \psi$ such that
$p-q \ge p^{1-\tilde{\epsilon}}$ for any small constant
$\tilde{\epsilon}> 0$ as long as $p$ is sufficiently large relative to
$q$.  

Next, if $J \supsetneq J_0$, then~\ref{ass:A3} and
\ref{ass:A4} allow us to relate $H_J(\hat{\beta}_J)$ and
$H_{J_0}(\hat{\beta}_{J_0})$ to the respective Hessian at the true
parameter, i.e., $H_{J} ({\beta}_{0})$ and $H_{J_0} (\beta_0)$.  We
find that
\begin{multline*}
\log\left( \frac{\det H_J (\hat{\beta}_J)}{\det H_{J_0}
    (\hat{\beta}_{J_0})} \right) \;\ge \;
 \log \left( \frac{\det H_J(\beta_0)}{\det H_{J_0} (\beta_0)}\right) +
 |J| \log \left( 1-
   \frac{c_{\mathrm{change}}}{c_{\mathrm{lower}}}\|\hat{\beta}_{J}
 - \beta_0\|_2 \right) \\ 
  - |J_0| \log  \left( 1+ \frac{c_{\mathrm{change}}}{c_{\mathrm{lower}}}\|\hat{\beta}_{J_0} - \beta_0\|_2 \right).
\end{multline*}
Note that by assuming $n$ large enough, we may assume that
$\|\hat{\beta}_J -\beta_0\|_2$ and $\|\hat{\beta}_{J_0} -\beta_0\|_2$
are small enough for the logarithms to be well defined;
recall~(\ref{eq:consist-true-MLE}).  Using that
$x\ge \log (1+x) \; \text{for all} \; x>-1$ and
$\log(1- \frac{x}{2}) \ge -x \; \text{for all} \; 0\le x \le 1$, we
see that 
\begin{multline*}
\log\left( \frac{\det H_J (\hat{\beta}_J)}{\det H_{J_0}
    (\hat{\beta}_{J_0})} \right) \;\ge \;
 \log \left( \frac{\det H_J(\beta_0)}{\det H_{J_0} (\beta_0)}\right)
  -2
  |J|\frac{c_{\mathrm{change}}}{c_{\mathrm{lower}}}\|\hat{\beta}_{J} 
- \beta_0\|_2 \\
  - |J_0| \frac{c_{\mathrm{change}}}{c_{\mathrm{lower}}}\|\hat{\beta}_{J_0} - \beta_0\|_2.
\end{multline*}
Under our scaling assumptions, $q^3\log(p)=o(n)$, and thus
applying~(\ref{eq:consist-true-MLE}) twice shows that
\begin{equation*}
-2 |J|\frac{c_{\mathrm{change}}}{c_{\mathrm{lower}}}\|\hat{\beta}_{J} - \beta_0\|_2   - |J_0| \frac{c_{\mathrm{change}}}{c_{\mathrm{lower}}}\|\hat{\beta}_{J_0} - \beta_0\|_2
\end{equation*}
is larger than any small negative constant for $n$ large enough.
For simplicity, we take the lower bound as $-1$.  By~\ref{ass:A3}, it
holds that
\begin{equation*}
\begin{split}
 \log \left( \frac{\det H_J(\beta_0)}{\det H_{J_0} (\beta_0)}\right)
 &= \log \det \left( H_{J\backslash J_0} (\beta_0)- H_{{J_0,{J\backslash J_0}}}(\beta_0)^{T} H_{J_0} (\beta_0)^{-1} H_{{J_0,{J\backslash J_0}}} (\beta_0)   \right)\\
  &\ge |J\backslash J_0| \log (n) + |J\backslash J_0| \log (c_{\mathrm{lower}}),
\end{split}
\end{equation*}
because the eigenvalues of the Schur complement of $H_J(\beta_0)$ are
bounded the same way as the eigenvalues of $H_J(\beta_0)$; see,
e.g.,~Chapter 2 of \cite{zhang2006schur}. Hence, for sufficiently
large $n$, the following is true for all $J\supsetneq J_0$:
\begin{equation}
  \label{eq:consist-hess-true}
\log \det H_J (\hat{\beta}_J)  -\log \det H_{J_0} (\hat{\beta}_{J_0})
\; \ge\;  |J\backslash J_0| \log (n) + |J\backslash J_0| \log
(c_{\mathrm{lower}} ) -1 .
\end{equation}

Combining the bound for the model prior probabilities
with~(\ref{eq:consist-fact-a}) and~(\ref{eq:consist-hess-true}), we
have for any true model $J\supsetneq J_0$ that
\begin{align*}
\log& \frac{\text{Laplace}_{\gamma}(J_0)}{\text{Laplace}_{\gamma}(J)}\\
&\ge\; -(1+\epsilon) (|J\backslash J_0|+\nu) \log (p) + \gamma
|J\backslash J_0| (1-\tilde{\epsilon}) \log (p)  + \frac{1}{2}
|J\backslash J_0| \log (n)\\
&\qquad - \gamma |J\backslash J_0| \log (q)
 + \frac{1}{2}|J\backslash J_0| \left(\log \frac{c_{\mathrm{lower}}}{{2\pi} }\right)+ \log \left( \frac{F_4}{F_3}\right) -1\\
&\ge\; \frac{1}{2} |J\backslash J_0| \left(\log (n) - \log q^{2\gamma}
  +2\left[(1-\tilde{\epsilon}) \gamma - (1+\epsilon)(1+\nu)
  \right]\log (p)  + \log \left(\frac{c_{\mathrm{lower}}
  }{2\pi}\right)     \right) \\
 &\qquad  + \log \left(\frac{F_4}{F_3} \right) -1.
\end{align*}
This lower bound is is positive for all $n$ large because our
assumption that $p = n^{\kappa}$, $q= n^{\psi}$ for
$0\le \psi< 1/3$, and
\[
\gamma > 1 - \frac{1-2\psi}{2(\kappa-\psi)}
\]
implies that
\begin{equation}
\label{Eq:condition}
\lim_{n\to\infty} \frac{\sqrt{n}}{p^{(1+\epsilon)(1+\nu) - \gamma
    (1-\tilde{\epsilon})} q^{\gamma}} = \infty
\end{equation}
provided the constants $\epsilon$, $\nu$, and $\tilde{\epsilon}$ are
chosen sufficiently small.  
\end{proof}

\section{Discussion}
\label{sec:discussion}

In this paper, we have shown that in the context of high-dimensional variable
selection problems, the Laplace approximation can be accurate
uniformly across a potentially very large number of sparse models.
We then showed how this approximation result allows one to give
results on the consistency of fully Bayesian techniques for variable
selection.

In practice, it is of course infeasible to evaluate the evidence or
Laplace approximation for every single sparse regression model, and some
search strategy must be adopted instead.  Some related numerical experiments
can be found in
\cite{chen2008extended}, \cite{chen2012extended},
\cite{zak2011modified}, and \cite{Rina2014Ising},
although that work considers BIC scores that drop some of the terms
appearing in the Laplace approximation.

Finally, we emphasize that the setup we considered concerns
generalized linear models without dispersion parameter and with
canonical link.  The conditions from \cite{LuoChen2013} could likely
be used to extend our results to other situations.

\bibliography{reference}

\appendix
\section{Technical lemmas}
\label{sec:Appendix B}
This section provides two lemmas that were used in the proof of Theorem~\ref{thm:LaplaceApproximation}.

\begin{lemma}[Chi-square tail bound]
\label{lemma:chi-square tail bound}  
Let $\chi^2_{k}$ denote a chi-square random variable with $k$
degrees of freedom.  Then, for any $n\ge 3$,
\[
\mathbb{P} \left\{ \chi^2_{k} \le 5k\log (n)    \right\} \ge 1-\frac{1}{n^{k}} \ge \exp\{-1/\sqrt{n}\}.
\]
\end{lemma}
\begin{proof}
  Since $\log(n)\ge 1$ when $n\ge 3$, we have that 
  \[
  k + 2\sqrt{k\cdot k\log (n)} + 2k \log (n) \;\le\;  5 k \log (n).
  \]
  Using the chi-square tail bound in \citet{laurent2000adaptive}, it
  thus holds that
  \begin{align*}
      \mathbb{P} \left\{\chi^2_{k} \le 5k\log (n) \right\} &\ge \mathbb{P} \left\{ \chi^2_{k} \le k + 2\sqrt{k\cdot k\log (n)} + 2k \log (n)  \right\} \\
      &\ge 1-e^{-k \log (n)}.
  \end{align*} 
  Finally, for the last step, by the Taylor series for $x\mapsto e^x$, for all $n\geq 3$ we have
  \[\exp\{-1/\sqrt{n}\}\leq 1 - \frac{1}{\sqrt{n}}+\frac{1}{2}\cdot\frac{1}{n} \leq 1 - \frac{1}{n}\;.\]
\end{proof}

\begin{lemma}
\label{lemma:integral bound}
Let $k\ge 1$ be any integer, and let $a,b >0$ be such that $ab\ge
2(k-1)$.  Then
\[
\int_{\| \xi\|_2 > a } \exp \{ -b\|\xi\|_2 \} d\xi \;\le\;
\frac{4(\pi)^{k/2}}{\Gamma\left( \frac{1}{2} k\right)}
\frac{a^{k-1}}{b} e^{-ab},
\]
where the integral is taken over $\xi \in \mathbb{R}^k$.
\end{lemma}

\begin{proof}
  We claim that the integral of interest is 
  \begin{align}
    \label{eq:polar}
    \int_{\| \xi\|_2 > a } \exp \{ -b\|\xi\|_2  \} d\xi 
    &= \frac{2(\pi)^{k/2}}{b^k \Gamma\left( \frac{1}{2}
        k\right)} \,\int_{r=ab}^\infty r^{k-1} e^{-r} dr.
  \end{align}
  Indeed, in $k=1$ dimension, 
  \begin{align*}
    \int_{\| \xi\|_2 > a } \exp \{ -b\|\xi\|_2  \} d\xi 
    &= 2\,\int_{r=a}^\infty  e^{-br} dr 
    = \frac{2}{b} e^{-ab},
  \end{align*}
  which is what (\ref{eq:polar}) evaluates to.
  If $k\ge 2$, then using polar coordinates
  \cite[Exercises 7.1-7.3]{Anderson:2003}, we find that
  \begin{align*}
    \int_{\| \xi\|_2 > a } \exp \{ -b\|\xi\|_2  \} d\xi 
    &= 2\pi \,\int_{r=a}^\infty r^{k-1} e^{-br} dr \cdot \prod_{i=1}^{k-2}
      \int_{-\pi/2}^{\pi/2} \cos^{i}(\theta_i) d\theta_i\\
    &= 2\pi \,\int_{r=a}^\infty r^{k-1} e^{-br} dr \cdot \prod_{i=1}^{k-2}
      \frac{\sqrt{\pi}\;\Gamma\left( \frac{1}{2} (i+1)
      \right)}{\Gamma\left( \frac{1}{2}
      (i+2)\right)}, 
  \end{align*}
  which again agrees with the formula from~(\ref{eq:polar}).  

  Now, the integral on the right-hand side
  of~(\ref{eq:polar}) defines the upper incomplete Gamma function and
  can be bounded as
  \begin{align*}
    \Gamma(k,ab) 
    &= \int_{r=ab}^\infty r^{k-1} e^{-r} dr
    \le 
    2 e^{-ab} (ab)^{k-1}
  \end{align*}
  for $ab\ge 2(k-1)$; see inequality (3.2) in \cite{Natalini:2000}.
  This gives the bound that was to be proven.
\end{proof}

\end{document}